\documentclass[a4paper]{amsart}
\usepackage[latin9]{inputenc}
\usepackage{a4}
\usepackage[all]{xy}
\usepackage{amsfonts}
\usepackage{amssymb}
\usepackage{amsmath}
\usepackage{amsthm}
\usepackage{changepage}
\usepackage{nomencl}
\usepackage{geometry}
\usepackage{tikz}
\usepackage{caption}
\usepackage{verbatim}
\usetikzlibrary{matrix}
\usepackage[toc,page]{appendix}
\begin{document}
\providecommand{\keywords}[1]{\textbf{\textit{Keywords: }} #1}
\newtheorem{thm}{Theorem}[section]
\newtheorem{lemma}[thm]{Lemma}
\newtheorem{prop}[thm]{Proposition}
\newtheorem{cor}[thm]{Corollary}
\theoremstyle{definition}
\newtheorem{defi}[thm]{Definition}
\theoremstyle{remark}
\newtheorem{remark}[thm]{Remark}
\newtheorem{prob}[thm]{Problem}
\newtheorem{conjecture}[thm]{Conjecture}
\newtheorem{ques}[thm]{Question}

\newcommand{\cc}{{\mathbb{C}}}   
\newcommand{\ff}{{\mathbb{F}}}  
\newcommand{\nn}{{\mathbb{N}}}   
\newcommand{\qq}{{\mathbb{Q}}}  
\newcommand{\rr}{{\mathbb{R}}}   
\newcommand{\zz}{{\mathbb{Z}}}  
\newcommand{\K}{\mathbb{K}}
\newcommand{\ra}{\rightarrow}
\newcommand{\Gal}{\operatorname{Gal}}
\newcommand{\ord}{\operatorname{ord}}
\newcommand{\Fix}{\operatorname{Fix}}
\newcommand{\ab}{\operatorname{ab}}
\newcommand{\Aut}{\operatorname{Aut}}
\newcommand{\PSL}{\operatorname{PSL}}
\newcommand{\AGL}{\operatorname{AGL}}
\newcommand{\SL}{\operatorname{SL}}

\title{Covering groups of $M_{22}$ as regular Galois groups over $\mathbb{Q}$}
\author{Joachim K\"onig}
\address{Department of Mathematics Education, Korea National University of Education, Cheongju, South Korea}
\email{jkoenig@knue.ac.kr}
\begin{abstract}
We close a gap in the literature by showing that the full 
covering groups of the Mathieu group $M_{22}$ and of its automorphism group $\Aut(M_{22})$ 
occur as the Galois groups of $\mathbb{Q}$-regular Galois extensions of $\mathbb{Q}(t)$. 
\end{abstract}
\maketitle

\section{Introduction and main result}
The sporadic Mathieu groups are a source for various headaches in inverse Galois theory.
For example, $M_{23}$ is notorious for being the only remaining sporadic group not yet known to occur as a Galois group over $\mathbb{Q}$ (see, e.g., \cite[Chapter II, Thm.\ 10.3]{MM}), in particular surviving all attempts with the classical rigidity criteria, as well as some ``brute force" calculations (e.g., \cite{KoenigPhD}).
That the group $M_{22}$ also leads to open problems in inverse Galois theory may be somewhat less well-known. The group itself (as well as its automorphism group) have long been known to occur as regular Galois groups (i.e., Galois groups of $\mathbb{Q}$-regular extensions of $\mathbb{Q}(t)$), see \cite{Malle88}. However, the Schur multiplier of $M_{22}$ is unusually large, being cyclic of order $12$, and the regular inverse Galois problem for the group $12.M_{22}$ (as well as some of its quotients) seems to be open, see, e.g., \cite[Chapter IV, Thm.\ 7.14]{MM}. In fact, the only non-trivial covering group of $M_{22}$ realized regularly in the literature seems to be $3.M_{22}$ (due to Feit \cite{Feit}, cf.\ \cite[Chapter IV, Thm.\ 5.8]{MM}).

The main result of this paper is the following.
\begin{thm}
\label{thm:main}
The full covering groups $12.M_{22}$ of $M_{22}$ and $2.\Aut(M_{22})$ of $\Aut(M_{22})$ occur as the Galois group of a $\mathbb{Q}$-regular Galois extension of $\mathbb{Q}(t)$.
\end{thm}

See Theorems \ref{thm:3point} and \ref{thm:3pointd} for a more detailed version of the result.
The proof is achieved via a combination of theoretical criteria and explicit computations. In particular, we give some new explicit polynomials with Galois group $\Aut(M_{22})$ and $M_{22}$ over $\mathbb{Q}(t)$, and investigate the behavior of their specializations. These are then used to show that certain embedding obstructions vanish.

In the assertion of Theorem \ref{thm:main}, some attention should be paid to the fact that, while the full covering group of a perfect group is unique up to isomorphism, a non-perfect group can have more than one full covering group. This happens, e.g., for the symmetric groups $S_n$ ($n\ge 5$, $n\ne 6$) which have two non-isomorphic double covers; and indeed, one may verify computationally that the same happens for $\Aut(M_{22})$. It will be clear though from the proofs (which use only the uniqueness of the stem covers of $M_{22}$) that all assertions hold for {\it both} stem extensions $2.\Aut(M_{22})$, whence we will refrain from introducing further notation for the two cases.

Note that the first explicit realizations with Galois group $M_{22}$ and $\Aut(M_{22})$ were presented by Malle in \cite{Malle88}; these remained the only ones for quite a while, until \cite{Koenig} gave new ones (in particular providing the first totally real realizations for these two groups). Another intriguing way to obtain $\Aut(M_{22})$-polynomials should at least be mentioned once: In \cite{BW} (Theorem 3.1 and Section 4.1.7), explicit rational functions $t=t(x)\in \mathbb{Q}(x)$ of degree $100$ (corresponding to degree-$100$ polynomials $F(t,x) = f(x)-tg(x)$) with monodromy group $\Aut(HS)$, the automorphism group of the sporadic Higman-Sims group, were computed. Since the point stabilizer in this group is isomorphic to $\Aut(M_{22})$, this means that, if $L/\mathbb{Q}(t)$ is the splitting field of such a polynomial $F(t,x)$, then $L/\mathbb{Q}(x)$ has Galois group $\Aut(M_{22})$.

{\bf Acknowledgement}:\\
I would like to thank Gunter Malle for helpful comments. 
This work was supported by the 2020 New Professor Research Grant funded by Korea National University of Education.

\section{Prerequisites}
\subsection{$M_{22}$ and its covering groups}
There was initially some confusion about the Schur cover of the Mathieu group $M_{22}$. In \cite{BF}, it was claimed that this maximal cover was $3.M_{22}$, and in a correction by the same authors, this was changed to $6.M_{22}$. Finally, in \cite{Mazet}, it was correctly identified as $12.M_{22}$. The Schur multiplier of $\Aut(M_{22})$ is of order $2$, as shown in \cite{Humphreys}. Note that, e.g., the triple cover $3.M_{22}$ can be extended to $3.\Aut(M_{22})$, but here the order $3$ normal subgroup is no longer central. 

For our purposes, in particular in view of Proposition \ref{prop:voe} and Lemma \ref{lem:prime_kernel} below, it is useful to know which subgroups $U$ of $M_{22}$ split in a covering group $C.M_{22}$ (with $C\le C_{12}$), in the sense that there exists a subgroup $\tilde{U}$ of $C.M_{22}$ mapping bijectively to $U$ under projection to $M_{22}$. We will only require the following (somewhat random-looking) results. These were verified with Magma (see \cite{Magma}), with code for all covering groups of $M_{22}$ available at \\ \texttt{http://brauer.maths.qmul.ac.uk/Atlas/v3/spor/M22/}.

\begin{prop}
\label{prop:split_subgps}
\begin{itemize}
\item[a)]
All subgroups of order $6$ of $M_{22}$ are split in $12.M_{22}$.
\item[b)]
All subgroups of order $1344$ in $M_{22}$ are split in $2.M_{22}$.
\end{itemize}
\end{prop}

\subsection{Specialization of function field extensions}

Let $K$ be a field and $F/K(t)$ a finite Galois extension. Assume without loss of generality that $F$ and $\overline{K}$ lie in some common overfield. The extension $F/K(t)$ is called {\it $K$-regular} (or simply {\it regular}, if the base field is clear), if $F\cap \overline{K} = K$, where $\overline{K}$ denotes the algebraic closure of $K$. For any $t_0\in K\cup\{\infty\}$ and any place $\mathfrak{p}$ of $F$ extending the $K$-rational place $t\mapsto t_0$, we have a residue field extension $F_{t_0}/K$. This is a Galois extension, not depending on the choice of place $\mathfrak{p}$. We call it the {\it specialization} of $F/K(t)$ at $t_0$.

Now let $K$ be of characteristic zero, and let $F/K(t)$ be a $K$-regular Galois extension with group $G$. 
For each $t_i\in \overline{K}\cup\{\infty\}$, the {\it ramification index} of $F/K(t)$ at $t_i$ is the minimal positive integer $e_i$ such that $F$ embeds into $\overline{K}(((t-t_i)^{1/e_i}))$. Note that, if $t_i=\infty$, one should replace $t-t_i$ by $1/t$. If the ramification index is larger than $1$, then $t_i$ is called a {\it branch point} of $F/K(t)$.
The set of branch points is always a finite set.
Associated to each branch point $t_i$ (of ramification index $e_i$) is a unique conjugacy class $C_i$ of $G$, corresponding to the automorphism $(t-t_i)^{1/e_i}\mapsto \zeta (t-t_i)^{1/e_i}$ of 
$\overline{K}(((t-t_i)^{1/e_i}))$, where $\zeta$ is a primitive $e_i$-th root of unity. The ramification index then equals the order of elements in the class $C_i$. The class tuple $(C_1,\dots, C_r)$ for all branch points $t_1,\dots, t_r$ of $F/K(t)$ is called the {\it ramification type} of $F/K(t)$.

In order to proceed from regular realizations (over $\mathbb{Q}(t)$) to realizations over $\mathbb{Q}$ with certain properties, we will make use of the following well-known theorem relating inertia groups of a $K$-regular Galois extension with those in its specializations (cf.\ \cite[Theorem 1.2 and Prop.\ 4.2]{Beckmann}).

\begin{thm}
\label{thm:beck}
Let $K$ be a number field and $N/K(t)$ a $K$-regular Galois extension with Galois group $G$.
Then with the exception of finitely many primes, depending only on $N/K(t)$, the following holds for every prime $\mathfrak{p}$ of $K$.\\
If $t_0\in K$ is not a branch point of $N/K(t)$, then the following condition is necessary for $\mathfrak{p}$ to be ramified in the specialization $N_{t_0}/K$:
 $$\nu_i:=I_{\mathfrak{p}}(t_0, t_i)>0 \text{ for some (automatically unique, up to algebraic conjugates) branch point $t_i$.}$$
 Here $I_{\mathfrak{p}}(t_0,t_i)$ is the intersection multiplicity of $t_0$ and $t_i$ at the prime $\mathfrak{p}$.
Furthermore, $\nu_i>0$ implies that the inertia group of a prime extending $\mathfrak{p}$ in $N_{t_0}/K$ is conjugate in $G$ to $\langle\tau^{\nu_i}\rangle$, where $\tau$ is a generator of an inertia subgroup over the branch point $t\mapsto t_i$ of $N/K(t)$.
\end{thm}

A direct consequence of Theorem \ref{thm:beck} is the following.
\begin{cor}
\label{cor:beck}
Assume $N/K(t)$ has a specialization $N_{t_0}/K$ unramified outside a set $S$ of primes of $K$. Then there exist infinitely many specializations $N_{t_1}/K$ such that all ramified primes $\mathfrak{p}$ of $N_{t_1}/K$, with the possible exception of the ones in $S$,  behave as in Theorem \ref{thm:beck}; in particular, the inertia group at $\mathfrak{p}$ in $N_{t_1}/K$ is contained in the inertia group at some $t_i$ in $N/K(t)$. Furthermore, we may choose $t_1$ such that none of those primes $\mathfrak{p}$ for which $t_0$ has non-negative $\mathfrak{p}$-adic valuation have $t_i=\infty$ in Theorem \ref{thm:beck}.
\end{cor}
\begin{proof}
This simply requires to choose $t_1$ in a suitable $S_0$-adic neighborhood of $t_0$, where $S_0$ is the set of exceptional primes for Theorem \ref{thm:beck}, in order to ensure (via Krasner's lemma) that $N_{t_0}/K$ and $N_{t_1}/K$ have the same local behavior at all primes in $S_0$. The second assertion follows since we may then choose $t_1$ $\mathfrak{p}$-integral as soon as $t_0$ is $\mathfrak{p}$-integral, whence it will not meet $t_i=\infty$ at $\mathfrak{p}$.
\end{proof}

\subsection{Embedding problems}
A {\it finite embedding problem} over a field $K$ is a pair $(\varphi:G_K\ra G,\varepsilon:\tilde{G} \ra G)$, where $\varphi$ is a (continuous) epimorphism from the absolute Galois group $G_K$ of $K$ onto $G$, 
and $\varepsilon$ is an epimorphism of finite groups. The kernel $\ker(\varepsilon)$ is called the kernel of the embedding problem. An embedding problem is called {\it central} if $\ker(\varepsilon) \le Z(\tilde{G})$, and {\it Frattini} if $\ker(\varepsilon)$ is contained in the Frattini subgroup of $\tilde{G}$.
A (continuous) homomorphism $\psi:G_K\ra \tilde{G}$ is called a \textit{solution} to $(\varphi, \varepsilon)$ if the composition  $\varepsilon \circ \psi$ equals $\varphi$. 
In this case, the fixed field of $\ker(\psi)$ is called a {\it solution field} to the embedding problem.
A solution $\psi$ is called a \textit{proper solution} if it is surjective. In this case, the field extension of the solution field over $K$ has full Galois group $\tilde{G}$. 
Given a finite vector ${\bf t} = (t_1,\dots, t_r)$ of independent transcendentals, an embedding problem $(\varphi:G_K\to G, \varepsilon)$ can be lifted to an embedding problem $(\varphi^\star: G_{K({\bf t})}\to G, \varepsilon)$ over $K({\bf t})$ by identifying the Galois group $G = \Gal(F/K)$ (where $F=\Fix(\ker(\varphi))$) with $\Gal(F({\bf t})/K({\bf t}))$.
A solution $\psi$ to the latter embedding problem is called {\it regular} if $F$ is algebraically closed in the fixed field of $\ker(\psi)$.

If $K$ is a number field and $\mathfrak{p}$ is a prime of $K$, every embedding problem $(\varphi, \varepsilon)$ induces an associated {\it local embedding problem} $(\varphi_{\mathfrak{p}}, \varepsilon_{\mathfrak{p}})$ defined as follows: $\varphi_{\mathfrak{p}}$ is the restriction of $\varphi$ to $G_{K_{\mathfrak{p}}}$ (well defined up to fixing an embedding of $\overline{K}$ into $\overline{K_{\mathfrak{p}}}$), and $\varepsilon_{\mathfrak{p}}$ is the restriction of $\varepsilon$ to $\varepsilon^{-1}(G({\mathfrak{p}}))$, where $G(\mathfrak{p}) := \varphi_{\mathfrak{p}}(G_{K_{\mathfrak{p}}})$.

We next recall some theoretical criteria for the solvability of certain embedding problems, which we will use in the following sections.
The first one deals with regular solutions of central embedding problems over $\mathbb{Q}^{\ab}$. See \cite[Chapter IV, Thm.\ 7.11 and Cor.\ 7.12]{MM}, or \cite[Criterion 2.2]{MS}). It is based on a local-global principle for Brauer embedding problems due to Sonn (\cite{SonnA}) and its application to fields with projective absolute Galois group (\cite{Sonn}).

\begin{prop}[Malle, Sonn]
\label{prop:ms}
Let $G$ be a finite group, $\tilde{G}=H.G$  a central extension, and let $L/\mathbb{Q}^{\ab}(t)$ be a $\mathbb{Q}^{\ab}$-regular $G$-extension with ramification type $C=(C_1,\dots, C_r)$ such that for each branch point $t_i\in \mathbb{P}^1(\overline{\mathbb{Q}})$ of $L/\mathbb{Q}^{\ab}(t)$, except possibly $t_i=\infty$, one of the following holds, with $\sigma_i\in C_i$ an inertia group generator at $t\mapsto t_i$:
\begin{itemize}
\item[i)] $\gcd(|H|, \ord(\sigma_i))=1$,
\item[ii)] $C_G(\sigma_i) = \langle \sigma_i\rangle$.
\end{itemize}
Then $L/\mathbb{Q}^{\ab}(t)$ embeds into a $\mathbb{Q}^{\ab}$-regular $\tilde{G}$-extension.
\end{prop}

The following two results are used to obtain regular solutions to central, resp.\ center-free Frattini embedding problems over $\mathbb{Q}$. See \cite{Voe} (Main Theorem) for the first, and \cite{Feit} or \cite[Chapter IV, Prop.\ 5.7]{MM} for the second, respectively.

\begin{prop}[V\"olklein]
\label{prop:voe}
Let $K$ be any field of characteristic $0$, let $L/K(t)$ be a $K$-regular $G$-extension, $\tilde{G} = H.G$ a central extension, and let $C=(C_1,\dots, C_r)$ be the ramification type of $L/K(t)$. Choose an element $\sigma_i\in C_i$ for each $i=1,\dots, r$. Assume that both of the following hold:
\begin{itemize}
\item[i)] $\gcd(|H|, \ord(\sigma_i))=1$ for all $i=1,\dots, r$.
\item[ii)] There exists a subgroup $U$ of $G$ whose preimage in $\tilde{G}$ splits the central subgroup $H$, and such that the fixed field $L^U$ has a place of degree coprime to $|H|$.
\end{itemize}
Then $L/K(t)$ embeds into a $K$-regular $\tilde{G}$-extension.
\end{prop}

\begin{prop}[Feit]
\label{prop:feit}
Let $L/K(t)$ be a $K$-regular $G$-extension, and let $\tilde{G} = H.G$ be a Frattini extension of $G$ with $Z(G) = Z(\tilde{G}) = 1$.
Let $C=(C_1,\dots, C_r)$ be the ramification type of $L/K(t)$, and assume that $C$ consists of rational conjugacy classes\footnote{Rationality here means that an element $x$ of this class is conjugate to all powers $x^d$ with $d$ coprime to the order of $x$. In particular rationality for the lifted classes $\tilde{C}_i$ in the following is automatic if, e.g., they can be chosen of the same element order as $C_i$.} $C_i$ of $G$ ($i=1,\dots, r$).
Assume that there exist rational classes $\tilde{C}_i\subset \tilde{G}$ lifting the classes $C_i$ ($i=1,\dots, r$) and fulfilling each of the following.
\begin{itemize}
\item[i)] $\tilde{C_i}= \tilde{C_j} \Leftrightarrow C_i=C_j$ (for all $1\le i,j\le r$).
\item[ii)] Set $f_i = \frac{|\tilde{C_i}|}{|C_i|}$. Then $f_1=\dots = f_{r-2} = 1$, and $f_{r-1} = |H|$.
\end{itemize}
Then $L/K(t)$ embeds into a $K$-regular $\tilde{G}$-extension of $K(t)$.
\end{prop}

Finally, the following is especially adapted to central embedding problems with kernel of order $2$. It is contained (in a language of Brauer classes) in \cite{Mestre1}, and applied to obtain regular Galois realizations with group $\SL_2(7)$ and $2.M_{12}$ in \cite{Mestre2} (see also \cite{Klueners} for an application to the group $\SL_2(11)$).

\begin{prop}[Mestre]
\label{prop:hwi}
Let $K$ be a field of characteristic $0$, let $L/K(t)$ be a $K$-regular $G$-extension with at most five branch points, and let $(\varphi: G_{K(t)}\to G, \varepsilon: \tilde{G} = 2.G\to G)$ be a central embedding problem with kernel of order $2$. Assume that there exists at least one non-branch point $t_0\in K$ such that the embedding problem given by $\varphi_0: G_K\to \Gal(L_{t_0}/K)$ and (restriction of) $\varepsilon$ is solvable. Then there exists an extension $K(u)/K(t)$ of rational function fields of degree $\le 16$ such that the embedding problem induced by $\tilde{\varphi}: G_{K(u)}\to \Gal(L(u)/K(u))$ and (restriction of) $\varepsilon$ is solvable. In particular, if $L(u)/K(u)$ is still $K$-regular with Galois group $G$, then it embeds into a $\tilde{G}$-extension.
\end{prop}

\begin{remark}
\label{rem:hwi}
The condition on $L(u)/K(u)$ remaining regular with group $G$ is automatically obtained in many situations, for example if $G$ has no proper subgroup of index $\le 16$, since then $L/K(t)$ and $K(u)/K(t)$ must be linearly disjoint even after base change to $\overline{K}$. Also, for the case of at most {\it four} branch points, Theorem 2 of \cite{Mestre0} mentions explicitly that $\overline{K}(u)/\overline{K}(t)$ may be chosen linearly disjoint from $L\overline{K}/\overline{K}(t)$, via making the branch point sets of the two extensions disjoint from each other. The analogous conclusion might be deduceable for the case of five branch points as well, although the preceding two observations will be sufficient for our purposes.
\end{remark}

\section{Regular Galois realizations with group $C.M_{22}.A$}

\label{sec:reg}
In \cite{MS} all covering groups of sporadic simple groups (and their automorphism groups), except the even degree covering groups of $M_{22}$ and $\Aut(M_{22})$ were realized regularly as Galois groups over $\mathbb{Q}^{\ab}(t)$. In fact, $3.M_{22}$ has been realized regularly over $\mathbb{Q}(t)$ by a criterion due to Feit (\cite{Feit}). 
\subsection{A three point $\Aut(M_{22})$-realization and its rational translates}
We begin by pointing out that all covering groups of $M_{22}$ and $\Aut(M_{22})$ can be realized regularly over $\mathbb{Q}^{\ab}$ using an $\Aut(M_{22})$-realization with three branch points corresponding to a rigid class triple. This realization may have been previously overlooked (cf.\ \cite{MS} or \cite[Chapter IV, Thm.\ 7.14]{MM}), even though it has been noticed in a different context (see Example 8.1 in Chapter III of \cite{MM}). 
However, the same realization yields much more, and we will indeed use it to realize the full covering groups $12.M_{22}$ and $2.\Aut(M_{22})$ regularly over $\mathbb{Q}$ in Section \ref{sec:mainres}. We give an explicit polynomial corresponding to the underlying $\Aut(M_{22})$-realization. 
Since this polynomial corresponds to a genus-$0$ triple in $\Aut(M_{22})$ (in the standard action on $22$ points) with three branch points, its computation is nowadays relatively standard.

\begin{thm}
\label{thm:3point}
Let $f(t,X) =(X^2-7X+15)^5 (X^2+15X+180)^5 (X^2+4X+400)- tX^6(X-4)^4$. Let $L$ be the splitting field of $f$ over $\mathbb{Q}(t)$.
Then \begin{itemize}
\item[a)] $L/\mathbb{Q}(t)$ has regular Galois group $\Aut(M_{22})$, and the ramification type with respect to $t$ is of the form $(2B,5A,12A)$ (in ATLAS notation).\footnote{Here, $2B$ is the class of elements of cycle type $2^7.1^8$ in $\Aut(M_{22})$. The other classes are uniquely identified by their element order.}
\item[b)] Setting $t=t(s) = \frac{3^3\cdot 5^4\cdot 11^{10}(s^2+55)}{2^8}$, the polynomial $f(t(s),X)$ has regular Galois group $M_{22}$ over $\mathbb{Q}(s)$, and the ramification type with respect to $s$ is of the form $(5A,5A,6A)$.
\item[c)] $L\mathbb{Q}^{\ab}/\mathbb{Q}^{\ab}(s)$ (resp., $L\mathbb{Q}^{\ab}/\mathbb{Q}^{\ab}(t)$) can be embedded into a $\mathbb{Q}^{\ab}$-regular Galois extension with group $12.M_{22}$ (resp., $2.\Aut(M_{22})$).
\end{itemize}
\end{thm}

\begin{proof}
Factoring first the discriminant of $f$ with respect to $x$, and then the polynomials $f(t_0,X)$ for the roots $t_0$ of the discriminant, one obtains that $L/\mathbb{Q}(t)$ is ramified exactly at three points with inertia group generators of cycle structure $(12.6.4)$, $(5^4.1^2)$, and $(2^7.1^8)$ respectively. Let $G=\Gal(f/\mathbb{Q}(t))$. It is easy to verify that $G$ is a doubly transitive subgroup of $S_{22}$; indeed, it suffices to factor the polynomial $f_1(X)f_2(Y)-f_2(X)f_1(Y)$, where $f_1$ and $f_2$ are the coefficients of $f$ at $t^0$ and $t^1$ respectively. The degrees of the irreducible factors of this polynomial correspond to the orbit lengths of a point stabilizer in $G$, and Magma confirms that these degrees are $1$ and $21$. Since there are no doubly transitive subgroups of $S_{22}$ other than $S_{22}$, $A_{22}$, $\Aut(M_{22})$ and $M_{22}$ (see, e.g., \cite{BL}),
 the above information on cycle structures in $G$ leaves only the possibilities $G=\Aut(M_{22})$ and $G=S_{22}$. Excluding the latter is not quite as easy, due to the high transitivity degree of $\Aut(M_{22})$.  
To obtain a strict proof without relying on any black-box methods, we follow a nice idea by Elkies (\cite{Elkies}). Namely, assume that $G=S_{22}$. Choose any prime $\lambda$ of good reduction for $f(t,X)$, and consider the degree-22 cover of projective curves $C\to \mathbb{F}_{\lambda}\mathbb{P}^1$ given by (the mod-$\lambda$ reduction of) $f(t,X)=0$. Good reduction implies that this cover is still regular with full Galois group $G=S_{22}$, and we may consider the subcover $C_4\to \mathbb{F}_{\lambda}\mathbb{P}^1$ of its Galois closure corresponding to the stabilizer of a $4$-set in $S_{22}$. The value $4$ is chosen due to the fact that, unlike $S_{22}$, the group $\Aut(M_{22})$ is not transitive on $4$-sets, which will eventually result in a contradiction. The curve $C_4$ is an absolutely irreducible $\mathbb{F}_\lambda$-curve, and its genus may be computed easily via the Riemann-Hurwitz genus formula, namely by computing the cycle structures of the three inertia group generators (i.e., elements of cycle structure $(12.6.4)$, $(5^4.1^2)$, and $(2^7.1^8)$) in the action on cosets of a $4$-set stabilizer in $S_{22}$ (here, the assumption $G=S_{22}$ is being used). Using Magma, this genus turns out to be $g=712$, and the Hasse-Weil bound yields that the number of $\mathbb{F}_\lambda$-points on $C_4$ is at most $\lambda + 1 + 2g\sqrt{\lambda}$. But on the other hand, points on $C_4$ can be found explicitly by exhaustive search: for any value $t_0\in \mathbb{F}_\lambda$ not dividing the discriminant of $f$, any degree-$4$ factor of $f(t_0,X)$ over $\mathbb{F}_\lambda$ corresponds to an $\mathbb{F}_\lambda$-point on $C_4$. Choosing $\lambda = 2160553$ (the $160000$-th prime), we find $4289839$ points in this way using Magma, whereas Hasse-Weil gives an upper bound of $4253666$. This contradiction shows that $G=\Aut(M_{22})$.

Next, to derive b), note that the fixed field $L^{M_{22}}$ in $L$ is quadratic over $\mathbb{Q}(t)$ and ramified at exactly two rational places. Hence it is a rational function field $\mathbb{Q}(s)$. Since $M_{22} = \Aut(M_{22})\cap A_{22}$, the concrete parameterization is obtained immediately via calculation of the discriminant.

Regarding c), the claim 
is a direct consequence of Proposition \ref{prop:ms}, upon noting that elements of order $12$ are self-centralizing in $\Aut(M_{22})$. Here we have used Proposition \ref{prop:ms} with the branch point of class $2B$ shifted to $t\mapsto \infty$, so that the finite branch points are of classes $5A$ and $12A$. This is of course without loss of generality up to a fractional linear transformation in the parameter $t$.
\end{proof}

\subsection{Specializations with prescribed behavior}
We now prepare the proof of Theorem \ref{thm:main} by providing specializations (in fact, infinitely many) of the polynomial in Theorem \ref{thm:3point} with Galois group $M_{22}$, for which the central embedding problem with kernel $C_4$ is solvable.

To do this, it is important to remember the following local-global principle for central embedding problems (cf.\ \cite[Chapter IV, Cor.\ 10.2]{MM}), which is essentially a consequence of the local-global principle for Brauer embedding problems over number fields (\cite[Chapter IV, Cor.\ 7.8]{MM}).

\begin{lemma}
\label{lem:prime_kernel}
Let $\Gamma = C.G$ be a central extension of $G$ by a cyclic group $C$ of prime order and $\varepsilon:\Gamma\to G$ the canonical projection.
Let $\varphi:G_{\mathbb{Q}}\to G$ be a continuous epimorphism. Then the embedding problem 
 $(\varphi, \varepsilon)$ is solvable if and only if all associated local embedding problems $(\varphi_p, \varepsilon_p)$ are solvable, where $p$ runs through all primes of $\mathbb{Q}$ (including the infinite one).
\end{lemma}

Since the local embedding problem at an unramified prime is always solvable, the above lemma gives an efficient method to check for global solvability of a given embedding problem by investigating only the finitely many ramified primes of $K/\mathbb{Q}$.

We also use the following criterion, see \cite[Prop.\ 2.1.7]{Serre}, in order to control ramification in solutions of embedding problems.
\begin{prop}
\label{prop:serre}
Let $\Gamma = C.G$ be a central extension of $G$ by a finite abelian group $C$, let $\varepsilon:\Gamma\to G$ be the canonical projection, and let $\varphi: G_{\mathbb{Q}}\to G$ be a continuous epimorphism such that the embedding problem $(\varphi,\varepsilon)$ has a solution. 
For each finite prime $p$, let $\tilde{\varphi_p}: G_{\mathbb{Q}_p}\to \Gamma$ be a solution of the associated local embedding problem $(\varphi_p,\epsilon_p)$, 
 chosen such that all but finitely many $\tilde{\varphi_p}$ are unramified. Then there exists a (not necessarily proper) solution $\tilde{\varphi}: G_{\mathbb{Q}}\to \Gamma$ of $(\varphi,\epsilon)$ such that for all finite primes $p$, the restrictions of $\tilde{\varphi}$ and $\tilde{\varphi_p}$ to the inertia group inside $G_{\mathbb{Q}_p}$ coincide. In particular, $\tilde{\varphi}$ is ramified exactly at those finite primes $p$ for which $\tilde{\varphi_p}$ is ramified.
\end{prop}

\begin{remark}
\label{rem:serre}
\begin{itemize}
\item[a)]
Concretely, $\tilde{\varphi}$ is obtained by twisting an initial solution by a suitable $C$-extension of $\mathbb{Q}$. Notably, in the special case $|C|=2$, if $F$ denotes the fixed field of $\ker(\varphi)$ and $F(\sqrt{x})$ is a solution field to the embedding problem (with $x\in F$), then all other solution fields are of the form $F(\sqrt{ax})$ with $a\in \mathbb{Q}^\times$. This is a special case of a more general result for Brauer embedding problems, cf.\ \cite[Chapter IV, Thm.\ 7.2]{MM}.
\item[b)] While Proposition \ref{prop:serre} does not a priori promise {\it proper} solutions (i.e., corresponding to Galois realizations with the full group $\Gamma$) to the embedding problem, this additional property is automatic in the case of stem extensions $C.G$ relevant for us, since in this case there is no proper subgroup projecting onto $G$. \end{itemize}
\end{remark}

We are now ready to state Theorem \ref{thm:spec_cover}, guaranteeing infinitely many specializations of the polynomial in Theorem \ref{thm:3point} whose splitting fields embed into $4.M_{22}$-extension.

\begin{thm}
\label{thm:spec_cover}
Let $g(s,X) =256(X^2-7X+15)^5 (X^2+15X+180)^5(X^2+4X+400)- 3^3\cdot 5^4\cdot 11^{10}(s^2+55)X^6(X-4)^4$ and $E/\mathbb{Q}(s)$ be a splitting field of $g$. Then among the specializations $E_{s_0}/\mathbb{Q}$, $s_0\in \mathbb{Q}$, there are infinitely many linearly disjoint ones whose Galois group equals $M_{22}$ and which embed into an extension $L/\mathbb{Q}$ with Galois group $4.M_{22}$. 
\end{thm}
\begin{proof}
We will identify, for each prime $p$ of $\mathbb{Q}$ (including the infinite one), a specialization value $s(p)\in \mathbb{Q}$ such that the embedding problem given by $G_{\mathbb{Q}}\to \Gal(E_{s(p)}/\mathbb{Q})$ and $4.M_{22}\to M_{22}$ induces a solvable local embedding problem at $p$.
To this aim, note first that for all rational primes $p$ other than $3$ and $\infty$, $g(s,X)$ has a specialization $g(s(p),X)$ whose splitting field $E_{s(p)}/\mathbb{Q}$ is unramified at $p$. Indeed, $s_0=35$ gives a field discriminant divisible only by $3, 5, 11$, whereas $s_0=11^{-5}$ gives an extension unramified at $11$ and $s_0=5^{-2}$ gives one unramified at $5$ (field discriminants were verified with Magma (\cite{Magma})). So for all these primes $p$ and corresponding values $s(p)$, local solvability is automatic.
For $p=\infty$, there is nothing to show since all order $2$ subgroups of $M_{22}$ are split by the central subgroup of order $4$ in $4.M_{22}$, whence the trivial solution of the local embedding problem is available (for any $s(p)\in \mathbb{Q}$).

Finally, for $p=3$, we claim that $s(p)=0$ yields a decomposition group of order dividing $6$ at $p$. Since all order $6$ subgroups of $M_{22}$ are split in $4.M_{22}$ as well (see Proposition \ref{prop:split_subgps}), the corresponding local embedding problem is then solvable via the trivial solution once again. To show the claim, note that $s=0$ is the unique point over the branch point of class $2B$ in the $\Aut(M_{22})$-extension $E/\mathbb{Q}(t)$ given by Theorem \ref{thm:3point}a). In particular, $E/\mathbb{Q}(s)$ is unramified at $s\mapsto 0$, with decomposition group $\Gal(E_0/\mathbb{Q})$ contained in $U:=C_{\Aut(M_{22})}(\sigma) \cap M_{22}$, where $\sigma\in 2B$. This intersection is a maximal subgroup of $M_{22}$ isomorphic to $\AGL_3(2) = 2^3.\PSL_3(2)$. Furthermore $E_0$ must contain all roots of $g(0,X)$. This polynomial has a degree-$8$ factor $\tilde{g}(X)=X^8 + 38X^7 + 783X^6 + 9795/2X^5 - 42189/2X^4 - 1312335/2X^3 +
    74121075/16X^2 - 23674275/2X + 12301875$ whose Galois group already equals $\AGL_3(2)$, so $E_0/\mathbb{Q}$ must equal the splitting field of $\tilde{g}(X)$. To gain information about the behavior of the prime $p=3$ in $E_0$, we may draw the Newton polygon of $\tilde{g}$, which is defined as the lower convex hull of the set of points $(j, \nu(c_j))$, $j=0,\dots, 8$, where $c_j$ denotes the coefficient of $\tilde{g}$ at $X^j$ and $\nu$ means $3$-adic valuation (see, e.g., \cite[Chapter II.6]{Neu} for background on the Newton polygon). This turns out to have line segments of length $4$, $3$ and $1$ with corresponding slopes  $-2$, $-1/3$ and $0$ respectively, see Figure \ref{fig:1}.
    
    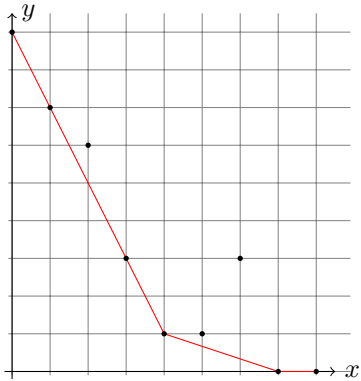
\begin{figure}[h!]
    \begin{tikzpicture}[scale=0.5]
    \draw[very thin,color=gray] (-0.1,-0.1) grid (8.9,9.5);    
    \draw [->](0,-0.2)--(0,9.5) node[right]{$y$};
\draw [->](-0.2,0)--(8.5,0) node[right]{$x$};
   \draw [thin,color=red] (0,9)--(4,1);
   \draw [thin,color=red] (4,1)--(7,0);
   \draw [thin,color=red] (7,0)--(8,0);
   \fill [black](0,9) circle(2pt);
   \fill [black](1,7) circle(2pt);
   \fill [black](2,6) circle(2pt);
   \fill [black](3,3) circle(2pt);
   \fill [black](4,1) circle(2pt);
   \fill [black](5,1) circle(2pt);
   \fill [black](6,3) circle(2pt);
   \fill [black](7,0) circle(2pt);
   \fill [black](8,0) circle(2pt);
   \end{tikzpicture}
\caption{Newton polygon for the polynomial $\tilde{g}$ at $p=3$.}
    \label{fig:1}
    \end{figure}

     It follows that the decomposition group of $E_0/\mathbb{Q}$ at $3$ (viewed as a subgroup of $\AGL_3(2)\le S_8$) has at least one fixed point (corresponding to the segment of length $1$) and at least one orbit of length $3$ (due to the slope $-1/3$). It turns out that the only subgroups of $\AGL_3(2)$ with such an orbit structure are isomorphic to $C_3$, $S_3$, $A_4$ or $S_4$. 
    But the latter two can never be decomposition groups at the prime $3$, since they do not have a nontrivial normal $3$-subgroup (taking the role of wild inertia) whose quotient is metacyclic. The decomposition group must therefore be $C_3$ or $S_3$. 
       This completes the proof of the claim. 
 
Now to show solvability of the global embedding problem, we choose specialization values $s_1\in \mathbb{Q}$ ``sufficiently close" $p$-adically to $s(p)$, for all the (finitely many) primes $p$ which are bad primes for the extension $E/\mathbb{Q}(s)$ in the sense of being exceptional for Theorem \ref{thm:beck}, and $s_1$ of non-negative $p$-adic valuation for all other primes $p$.
Krasner's lemma implies that for all bad primes $p$, the local behaviors of $E_{s(p)}/\mathbb{Q}$ and $E_{s_1}/\mathbb{Q}$ at $p$ are the same, and in particular, do not obstruct the embedding problem. By Theorem \ref{thm:beck} and Corollary \ref{cor:beck}, all {\it other} primes can only ramify in $E_{s_1}/\mathbb{Q}$ with an inertia group contained in the inertia group at a finite branch point of $E/\mathbb{Q}(s)$, i.e., one of order $5$. But it is easy to see that a central embedding problem with kernel $C_2$ is always locally solvable at $p$ as soon as the inertia group at $p$ is of odd order (indeed, if the local embedding problem is split, then there is nothing to show, and if it is non-split, then extending the maximal unramified $2$-extension by a degree $2$ yields a solution).
 
We now decompose the embedding problem with kernel $C_4$ into two subsequent problems with kernel $C_2$. By Lemma \ref{lem:prime_kernel}, the first one is solvable by the above. Furthermore, by Proposition \ref{prop:serre} together with the fact that all the inertia groups split in $2.M_{22}$, the solution may be chosen 
such that no finite primes ramify in the quadratic extension corresponding to the kernel

 Let $F\supset E_{s_1}\supset \mathbb{Q}$ be such a solution field (in particular, $\Gal(F/\mathbb{Q}) = 2.M_{22}$). It remains to show that the local extensions $F_p/\mathbb{Q}_p$ at all ramified primes $p$ in this first solution do not obstruct the second embedding step. 
 Since all ramified primes other than $p=3$ are of odd ramification index (namely, $5$), local solvability is still guaranteed as above. 
 For $p=3$, the decomposition group in $F/\mathbb{Q}$ is either contained in $S_3$ (which splits in $4.M_{22}$), or in $S_3\times C_2$ (which extends to $S_3\times C_4$ in $4.M_{22}$). In the latter case, by construction, the second factor $C_2$ must correspond to the full unramified extension, whence extending this quadratic unramified extension of $\mathbb{Q}_3$ to the degree-$4$ one solves the local embedding problem. 
 Hence the second embedding step is solvable as well.

Since the specialization values $s_1\in \mathbb{Q}$ admissible for the above form an $S$-adically open set for $S$ the finite set of bad primes of $E/\mathbb{Q}(s)$, the well-known compatibility of Hilbert's irreducibility theorem with weak approximation (e.g., \cite[Proposition 2.1]{PV}) instantly yields infinitely many linearly disjoint extensions $E_{s_1}/\mathbb{Q}$ with full Galois group $M_{22}$ and embedding into $4.M_{22}$-extensions.
\end{proof}

\subsection{The main result}
\label{sec:mainres}
We are ready to show the following concrete version of Theorem \ref{thm:main}.
\begin{thm}
\label{thm:3pointd}
Let $L$, $\mathbb{Q}(t)$ and $\mathbb{Q}(s)$ be as in Theorem \ref{thm:3point}. There exist 
 extensions $\mathbb{Q}(u)/\mathbb{Q}(s)$ and $\mathbb{Q}(w)/\mathbb{Q}(t)$ of rational function fields such that 
 \begin{itemize}
 \item[a)]
 $L(u)/\mathbb{Q}(u)$ is $\mathbb{Q}$-regular of group $M_{22}$ and can be embedded into a $\mathbb{Q}$-regular Galois extension with group $12.M_{22}$, the full covering group of $M_{22}$.
 \item[b)] $L(w)/\mathbb{Q}(w)$ is $\mathbb{Q}$-regular of group $\Aut(M_{22})$ and can be embedded into a $\mathbb{Q}$-regular Galois extension with group $2.\Aut(M_{22})$, \underline{any} full covering group of $\Aut(M_{22})$.
 \end{itemize}
\end{thm}
\begin{proof}
Let $\mathfrak{p}$ be the (degree $1$) place of $\mathbb{Q}(t)$ of ramification index $2$ in $L$, and $\mathfrak{q}$ the unique place of $\mathbb{Q}(s)$ extending $\mathfrak{p}$. Since the class $2B$ lies outside of $M_{22}$, the extension $L/\mathbb{Q}(s)$ is unramified at $\mathfrak{q}$. Furthermore, the decomposition group at (any place extending) $\mathfrak{q}$ is contained in $U:=C_{\Aut(M_{22})}(\sigma) \cap M_{22}$, where $\sigma\in 2B$. As seen before, this intersection is a maximal subgroup of $M_{22}$ of order $1344$, isomorphic to $\AGL_3(2) = 2^3.\PSL_3(2)$. By Proposition \ref{prop:split_subgps}, the preimage of $U$ in $2.M_{22}$ splits the central subgroup $C_2$.
Now assume without loss that $\mathfrak{q}$ is the place $s\mapsto 0$, the unique ramified place of class $6A$ is $s\mapsto \infty$, and set $v:=\sqrt{as}$ (for any $a\in \mathbb{Q}^\times$ of our choice). Obviously $\mathbb{Q}(v)/\mathbb{Q}(s)$ is totally ramified of degree $2$ at $s\mapsto 0$ and $s\mapsto \infty$ (and linearly disjoint from $L/\mathbb{Q}(s)$). Thus, $L(v)/\mathbb{Q}(v)$ is still $\mathbb{Q}$-regular of group $M_{22}$, with four ramified places of ramification index $5$ and one of ramification index $3$, and with decomposition group at $v\mapsto 0$ still contained in $U$.
It thus follows from Proposition \ref{prop:voe} that $L(v)/\mathbb{Q}(v)$ embeds into a $\mathbb{Q}$-regular extension $F/\mathbb{Q}(v)$ with group $2.M_{22}$. 
We may assume that $F/\mathbb{Q}(v)$ is unramified outside of the set of branch points of $L(v)/\mathbb{Q}(v)$. Indeed, if it acquires extra ramification, then the following field crossing argument - a concrete case of the twisting in Remark \ref{rem:serre}a) - gets rid of this ramification by changing the quadratic extension of $L(v)$:

Let $M=\mathbb{Q}(v, \sqrt{p(v)})$, where $p(X)\in \mathbb{Q}[X]$ is the product of all linear factors $X-v_i$, where $v_i\in \overline{\mathbb{Q}}$ is a branch point of $L(v)/\mathbb{Q}(v)$ over which $F/L(v)$ is ``newly" ramified. Define $\tilde{L}$ to be the compositum of $F$ and $M$. This has Galois group $2.M_{22}\times C_2$. Inside the socle $C_2\times C_2$ of this group, let $\Delta$ be the diagonal subgroup, and let $\tilde{L}^\Delta$ be its fixed field. Then $\tilde{L}^\Delta/\mathbb{Q}(v)$ is Galois with group $2.M_{22}$. Furthermore, all newly ramified places $v\mapsto v_i$ have inertia group $\Delta$ inside $\tilde{L}/\mathbb{Q}(v)$, and are thus unramified in $\tilde{L}^\Delta/\mathbb{Q}(v)$. Replace $F$ by $\tilde{L}^\Delta$, if necessary.

We may thus assume that $F/\mathbb{Q}(v)$ is a $\mathbb{Q}$-regular Galois extension of group $2.M_{22}$ and with five branch points. We now apply Proposition \ref{prop:hwi} with $G=2.M_{22}$ and $\Gamma = 4.M_{22}$. All that is required is to provide one specialization $v\mapsto v_0\in \mathbb{Q}$ for which $F_{v_0}/\mathbb{Q}$ has full Galois group and for which the induced embedding problem is solvable over $\mathbb{Q}$. This is done by Theorem \ref{thm:spec_cover}, which a priori provided such specializations $s\mapsto s_0$ for the $M_{22}$-extension $L/\mathbb{Q}(s)$. Note however, that we have the freedom to choose $a\in \mathbb{Q}^\times$ in the definition of $v=\sqrt{as}$, and we may now simply choose it such that $v_0:=\sqrt{as_0}$ is rational, in which case $L(v)_{v_0}/\mathbb{Q} = L_{s_0}/\mathbb{Q}$. So we know that $F/\mathbb{Q}(v)$ has a fiber $F_{v_0}/\mathbb{Q}$ factoring through an $M_{22}$-extension as in the proof of Theorem \ref{thm:spec_cover}, and upon twisting $F/\mathbb{Q}(v)$ (in the sense of Remark \ref{rem:serre}a)) once more by a suitable quadratic {\it constant} extension $\mathbb{Q}(\sqrt{d})(v)/\mathbb{Q}(v)$, we may in fact assume that the whole fiber $F_{v_0}/\mathbb{Q}$ equals a $2.M_{22}$-extension as in the proof of Theorem \ref{thm:spec_cover} (in particular, one which embeds into a $4.M_{22}$-extension due to Theorem \ref{thm:spec_cover}, since the constant twist does not change the branch point locus). From Proposition \ref{prop:hwi}, we obtain a rational function field $\mathbb{Q}(u)\supseteq \mathbb{Q}(v)$ such that $L(u)/\mathbb{Q}(u)$ (still has full Galois group $M_{22}$ and) embeds into a $4.M_{22}$-extension.

To obtain the same for $3.M_{22}$, we embed the $\Aut(M_{22})$-extension $L/\mathbb{Q}(t)$ into an extension with group $3.\Aut(M_{22})$ (a centerless Frattini extension of $\Aut(M_{22})$) using Proposition \ref{prop:feit} as follows. Ordering the classes in the ramification type of $F/\mathbb{Q}(a)(t)$ as $(5A,2B, 12B)$, we may verify with the help of Magma that these classes lift to (unique) classes of elements of the same orders $5, 2$ and $12$ in $3.\Aut(M_{22})$ and such that, in the notation of Proposition \ref{prop:feit}, $f_1=1$ and $f_2=f_3=3$. Indeed, the elements of order $5$ and $3$ necessarily centralize the order-$3$ normal subgroup whence the conjugacy class length does not grow, whereas the elements of class $2B$ and $12A$ do not centralize it. Proposition \ref{prop:feit} is now applicable, showing that $L/\mathbb{Q}(t)$ embeds into a $\mathbb{Q}$-regular $3.\Aut(M_{22})$-extension. In particular, $L/\mathbb{Q}(s)$, and a fortiori the translate $L(u)/\mathbb{Q}(u)$, embeds into a $3.M_{22}$-extension. Altogether, we obtain a $12.M_{22}$-extension of $\mathbb{Q}(u)$. Regularity follows since all non-trivial normal subgroups of $12.M_{22}$ are contained in the center. This completes the proof of a).

Obtaining b) is now easy. Indeed, the above specializations of $L/\mathbb{Q}(s)$ rendering the $2.M_{22}$ embedding problem solvable are also specializations of the $\Aut(M_{22})$-extension $L/\mathbb{Q}(t)$. That one has only three branch points, so the conclusion follows readily from Proposition \ref{prop:hwi} (together with Remark \ref{rem:hwi}).
\end{proof}

\begin{remark}
I do not know whether the above approach could also succeed in full with the $M_{22}$-polynomial given by Malle in \cite{Malle88}. One potential bother comes from the fact that for that polynomial, all $M_{22}$-specializations are wildly ramified at $2$ (see \cite[Proposition 2.5]{PV}), which could in principle lead to the embedding problem with kernel $C_4$ being obstructed at the prime $2$ in every fiber. Whether this is indeed the case, I have not verified.
\end{remark}

\end{document}